\documentclass[11pt,leqno]{article}
\usepackage{amsthm,amsfonts,amssymb,epsfig,graphics,amsmath}
\usepackage[latin1]{inputenc}\relax

\font\tenronde=rsfs10
\font\sevenronde=rsfs7
\font\fiveronde=rsfs5
\newfam\rondefam
\scriptscriptfont\rondefam=\fiveronde
\textfont\rondefam=\tenronde
\scriptfont\rondefam=\sevenronde

\newcommand\br{\begin{rem}}
\newcommand\er{\end{rem}}
\newcommand\bp{\begin{pmatrix}}
\newcommand\ep{\end{pmatrix}}
\newcommand\be{\begin{equation}}
\newcommand\ee{\end{equation}}
\newcommand\ba{\begin{equation}\begin{aligned}}
\newcommand\ea{\end{aligned}\end{equation}}
\newcommand\kernel{\hbox{\rm Ker}}

\newcommand{\csLd}{{\Psi^\eps(\tilde U)}}

\newcommand{\sgn}{{\text{\rm sgn}}}

\newcommand{\R}{\Re}
\newcommand{\e}{\epsilon}

\newcommand\s{\sigma}

\def\g{\gamma}

\newcommand{\RR}{{\mathbb R}}

\newcommand\cU{{\cal  U}}

\newcommand\cV{{\cal  V}}

\newcommand\cR{{\cal  R}}

\newcommand\cN{{\cal  N}}

\newcommand\cS{{\mathcal S}}

\newcommand{\mez}{{\frac{1}{2}}}

\def\eps{\varepsilon }

\def\D{\partial }
\newcommand\adots{\mathinner{\mkern2mu\raise1pt\hbox{.}
\mkern3mu\raise4pt\hbox{.}\mkern1mu\raise7pt\hbox{.}}}

\newcommand{\Id}{{\rm Id }}

\newcommand{\re}{{\rm Re }\, }

\newtheorem{theo}{Theorem}[section]
\newtheorem{prop}[theo]{Proposition}
\newtheorem{cor}[theo]{Corollary}
\newtheorem{lem}[theo]{Lemma}

\newtheorem{ass}[theo]{Assumption}

\newtheorem{rem}[theo]{Remark}

\numberwithin{equation}{section}

\begin{document}

\title { Existence of quasilinear relaxation shock profiles}

\author{\sc \small Guy M\'etivier\thanks{
IMB, Universit\'e de Bordeaux, CNRS, IMB, 
33405 Talence Cedex, France; metivier@math.u-bordeaux.fr.:
G.M. thanks Indiana University for its hospitality during a
visit in which this work was partly carried out,
},
Benjamin Texier \thanks{
Universit\'e Paris Diderot (Paris 7), Institut de Math\'ematiques de Jussieu, UMR CNRS 7586;
texier@math.jussieu.fr:
Research of B.T.  was partially supported
under NSF grant number DMS-0505780.
},
Kevin Zumbrun\thanks{Indiana University, Bloomington, IN 47405;
kzumbrun@indiana.edu:
K.Z. thanks the University of Bordeaux I
for its hospitality during the visit in which
this work was carried out.
Research of K.Z. was partially supported
under NSF grants number DMS-0070765 and DMS-0300487.
 }}


\maketitle

\begin{abstract}
We establish existence with sharp rates of decay and distance
from the Chapman--Enskog approximation
of small-amplitude quasilinear relaxation
shocks in the general case that the profile ODE may become degenerate.
Our method of analysis follows the general approach 
used by M\'etivier and Zumbrun in the semilinear case,
based on Chapman--Enskog expansion and
the macro--micro decomposition of Liu and Yu.
In the quasilinear case, however, 
we find it necessary to apply a parameter-dependent
Nash-Moser iteration to close the analysis,
whereas, in the semilinear case, 
a simple contraction-mapping argument sufficed.
\end{abstract}

\tableofcontents


\section{Introduction}\label{intro}

We consider the problem of existence of relaxation profiles 
\begin{equation}\label{relaxprof}
U(x,t)=\bar U(x-st), \quad \lim_{z\to \pm \infty}\bar U(z)=U_\pm
\end{equation}
of a general relaxation system 
\begin{equation}\label{grelax}
U_t +A(U)U_x= Q(U), 
\end{equation}
\begin{equation}\label{block}
\quad U= \begin{pmatrix} u\\v \end{pmatrix},
\quad A= \begin{pmatrix} A_{11}& A_{12}\\A_{21}& A_{22}\end{pmatrix},
\quad Q=\begin{pmatrix} 0\\q\end{pmatrix},
\end{equation}
in one spatial dimension, $u\in \RR^n$, $v\in \RR^r$, where, 
for some smooth $v_*$ and $f$,
\begin{equation}\label{qassum1}
q(u,v_*(u))\equiv 0, \quad  
\Re \sigma (\partial_v q(u,v_*(u)))\le -\theta,\; \theta>0,
\end{equation}
$\sigma(\cdot )$ denoting spectrum, and
\begin{equation}\label{Aform}
\quad \begin{pmatrix} A_{11}& A_{12}\end{pmatrix}= 
\begin{pmatrix} \partial_u f& \partial_v f\end{pmatrix}.
\end{equation}
Here, we are thinking particularly of the case $n$ bounded
and  $r \gg 1$ arising through discretization or moment closure
approximation of the Boltzmann equation or other kinetic models;
that is, we seek estimates and proof independent of the dimension
of $v$.

For fixed $n$, $r$, the existence problem has been treated in
\cite{YZ, MaZ1} under the additional assumption 
\begin{equation}\label{nondeg}
\det (A-sI)\ne 0
\end{equation}
corresponding to nondegeneracy of the traveling-wave ODE.
However, as pointed out in \cite{MaZ2,MaZ3}, this assumption is
unrealistic for large models, and in particular is not satisfied for
the Boltzmann equations, for which the eigenvalues of $A$ are
constant particle speeds of all values, hence cannot be uniformly
satisfied for discrete velocity or moment closure approximations.
Our goal here, therefore, is to revisit the existence problem
without the assumption \eqref{nondeg}.

The latter problem was treated in \cite{MZ2} for the semilinear case,
which includes discrete velocity approximations of Boltzmann's equations,
and for Boltzmann's equation (semilinear but infinite-dimensional) 
in \cite{MZ3}.
We mention also the proof, by similar methods, 
of positivity of Boltzmann shock profiles in \cite{LY} 
and the original proof, by different methods,
of existence of Boltzmann profiles in \cite{CN}.
The new application here is to moment closure approximations of Boltzmann's
and other kinetic equations, which are in general quasilinear.

Our main result is to show existence with sharp rates of decay and distance
from the Chapman--Enskog approximation of small-amplitude quasilinear relaxation
shocks in the general case that the profile ODE may become degenerate.
See Sections \ref{model} and \ref{CEapprox} for model assumptions and 
description of the Chapman--Enskog approximation, and 
Section \ref{results} for a statement of the main theorem.
Our method of analysis, as in \cite{MZ2,MZ3} is based on Chapman--Enskog
expansion and  the macro-micro decomposition of \cite{LY}.
The main difference in this analysis from those of the previous works
is that, due to a subtle loss of derivatives,
{\it in the quasilinear case, we find it necessary to apply Nash-Moser
iteration to close the analysis}, whereas in 
the semilinear case a simple contraction-mapping argument sufficed.\footnote{
See Remark \ref{whynm} for further discussion of this point.
}
Indeed, we require a nonstandard, 
parameter-dependent, Nash--Moser iteration scheme,
indexed by amplitude $\eps\to 0$, for which the linear solution operator
loses not only derivatives but powers of $\eps$.
In this, we make convenient use of a general scheme developed in \cite{TZ} for
the treatment of such problems, which also arise in certain 
weakly nonlinear optics problems involving oscillatory solutions
with large amplitudes or times of existence.

We note that spectral stability has been shown for general
small-amplitude quasilinear relaxation profiles in \cite{MaZ3}, without the
assumption \eqref{nondeg}, under the assumption that the profile
exist and satisfy exponential bounds like those of the viscous case.
The results obtained here verify that assumption, completing the 
analysis of \cite{MaZ3}.
Existence results in the absence of condition
\eqref{nondeg} have been obtained in special cases in \cite{MaZ4,DY}
by quite different methods
(for example, 
center-manifold expansion near an assumed single degenerate point \cite{DY}).
However, the decay bounds as stated,
though exponential, are not sufficiently
sharp with respect to $\eps$ for the needs of \cite{MaZ3}.

\section{Model, assumptions, and the reduced system}\label{model}

Taking without loss of generality $s=0$, we study the traveling-wave
ODE
\begin{equation}\label{relax}
A(U)U'= Q(U),
\end{equation}
\begin{equation}\label{relaxform}
\quad  U=\begin{pmatrix} u\\v \end{pmatrix},
\quad A=\begin{pmatrix} \partial_u f(u,v)& \partial_v f(u,v)\\A_{21}(u,v)
&A_{22}(u,v)\end{pmatrix},
\quad Q=\begin{pmatrix} 0\\q(u,v)\end{pmatrix}
\end{equation}
governing solutions \eqref{relaxprof}, where
\begin{equation}\label{qassum}
q(u,v_*(u))\equiv 0, \quad  \Re \sigma(\partial_v q(u,v_*(u)))
\le -\theta,\; \theta>0.
\end{equation}
We make the standard assumption of {\it symmetric--dissipativity}
\cite{Y}:

\begin{ass}\label{SS}
(SD) \quad  There exists a smooth,  symmetric and uniformly positive
definite matrix $S(U)$ such that 

\quad i)  for all $U$, $S(U)A$ is symmetric, 

\quad ii) for all equilibria $ U_*= (u, v_*(u))$,
$\Re S\,dQ(U_*)$ is nonpositive with
\begin{equation}\label{rank}
\dim \ker \Re SdQ=\dim \ker dQ \equiv n.
\end{equation}
\end{ass}

In \eqref{rank} and below, $\Re M$ denotes symmetric part of the matrix $M,$ i.e. $\frac{1}{2}(M + M^*).$ 

By the change of coordinates $v\to v-v_*(u,v)$, we may take
without loss of generality
\begin{equation}\label{zerov}
v_*(u,v)\equiv 0,
\quad
dQ=\begin{pmatrix} 0 & 0\\ 0 & \partial_v q\end{pmatrix}
\end{equation}
without changing either the assumed structure \eqref{grelax}, \eqref{relax}
or (since it is coordinate-independent) the property of symmetrizability.
Note that symmetry of $SdQ$, together with \eqref{rank}, 
then implies both block-diagonal structure
\begin{equation}\label{Sdi}
S=\begin{pmatrix} S_{11}& 0\\ 0 & S_{22}\end{pmatrix}
\end{equation}
and definiteness and proper rank of $\Re S_{22}\partial_v q$.
Likewise, symmetry of $SA$ together with \eqref{Sdi} yields
symmetry of $S_{11}A_{11}$ and $S_{22}A_{22}$ as well as
\begin{equation}\label{keyrel}
(S_{11} A_{12})^T= S_{22}A_{21}.
\end{equation}
We make the simplifying assumption \eqref{zerov} throughout the paper.

We make also the Kawashima assumption of {\it genuine coupling} \cite{K}:
 
 \begin{ass} \label{GC}
 \label{assGC}
(GC) \quad  For all equilibria $ U_*= (u, v_*(u))$,  there exists no eigenvector of 
$A$ in the kernel of $dQ(U_*)$.
Equivalently, given Assumption \ref{SS} (see \cite{K}), 
there exists in a neighborhood $\cN$ of the
equilibrium manifold a skew symmetric $K=K(U)$ such that 
\begin{equation}\label{K}
\Re (KA-SdQ)(U) \ge \theta>0
\end{equation}
for all $U\in \cN$.
\end{ass}

Recall \cite{Y} that the reduced, Navier--Stokes type equations
obtained by Chapman--Enskog expansions are
\begin{equation}\label{ce} 
f_*(u)'= (b_*(u)u')',
\end{equation}
where, under the simplifying assumption \eqref{zerov},
\begin{equation}\label{cevalues} 
\begin{aligned}
f_*(u)&:=f(u,0),\\
b_*(u)u'&:= -A_{12}\partial_v q^{-1}A_{21} (u,0).
\end{aligned}
\end{equation}
For the reduced system \eqref{ce}, 
symmetric--dissipativity becomes:
\medbreak
(sd)\quad
There exists $s(u)$ symmetric positive
definite such that $s\, df_*$ is symmetric and $sb_*$ is
symmetric positive semidefinite, with 
$\dim \ker \Re sb_*=\dim \ker b_* $.
\medbreak
\noindent 
We have likewise a notion of genuine coupling \cite{K}:

\medbreak
(gc)\quad 
There is no eigenvector of $df_*$ in $\ker b_*$.
\medbreak

We note first the following important observation of \cite{Y}.

\begin{prop}[\cite{Y}]\label{redconds}
Let \eqref{relax} as described above be a symmetric--dissipative
system satisfying the genuine coupling condition (GC).
Then, the reduced system \eqref{ce} is a symmetric--dissipative
system satisfying genuine coupling condition (gc).
\end{prop}

\begin{proof}
Assuming without loss of generality \eqref{zerov},
we find that $s=S_{11}$ is a symmetrizer,
since $sdf_*=S_{11}A_{11}$ is symmetric as already observed, 
and $sb_*= -S_{11}A_{12}(S_{22}\partial_v q)^{-1}S_{22}A_{21}$
is definite with proper rank by the corresponding properties of
$S_{22}\partial_v q$ together with \eqref{keyrel}. 
Computing that (gc) is the condition that no eigenvector
of $A_{11}$ lie in $\ker A_{21}$, we see that (GC) and (gc)
are equivalent.
\end{proof}

Besides the basic properties guaranteed by Lemma \ref{redconds},
we assume that the reduced system satisfy the following important
additional conditions.

\begin{ass}\label{goodred}
\textup{
(i) The matrix $b_*(u)$ has constant left kernel, with
associated eigenprojector $\pi_*$ onto $\ker b_*$, and
(ii) The matrix $a_*:=\pi_* df_* \pi_*(u)|_{\ker b_*}$ is uniformly
invertible.
}
\end{ass}

Assumption \ref{goodred} 
ensures that the zero-speed profile problem for the reduced system,
\begin{equation}\label{vprof}
f_*(u)'=(b_*(u)u')', \quad \lim_{z\to \pm \infty} u(z)=u_\pm 
\end{equation}
or, after integration from $-\infty$ to $x$,
\begin{equation}\label{intvprof}
b_*(u)u'= f_*(u)-f_*(u_\pm),
\end{equation}
may be expressed as a nondegenerate ODE in $u_2$, coordinatizing
$u=(u_1,u_2)$ with $u_1=\pi_*u$ and $u_2=(I-\pi_*)u$
\cite{MaZ3,Z1,GMWZ}. 
Next, we assume that the classical theory of    weak shocks 
can be applied to  \eqref{vprof}, assuming that the flux $f_*$ has a  genuinely nonlinear 
eigenvalue near $0$:

\begin{ass}\label{profass}
 
In a neighborhood $\cU_*$ of   a given base state $u_0$,  
$df_*$ has a simple eigenvalue $\alpha$ near zero, with $\alpha (u_0) = 0$, and such that the
associated hyperbolic characteristic field is genuinely 
nonlinear, i.e., after a choice of orientation, $\nabla \alpha \cdot r(u_0) <  0$, where
$r$ denotes the eigendirection associated with $\alpha$.
 
\end{ass}

\begin{rem}
\textup{Assumption \ref{profass} is standard,
and is satisfied in particular for the compressible
Navier--Stokes equations resulting from Chapman--Enskog approximation
of the Boltzmann equation. 
Assumptions \ref{SS} and \ref{GC} are verified in \cite{Y} 
for a wide variety of discrete kinetic models.\footnote{
For example, both discrete kinetic models \cite{PI}
used to approximate the Boltzmann equation \cite{PI} 
and BGK models \cite{JX,N} used to approximate general
hyperbolic conservation laws; see pp. 289--294 \cite{Y}.  
Note for each of these examples that the symmetrizer $S$ is
not constant, but depends nontrivially on $U$.}
Assumptions \ref{goodred} and \ref{profass} on the reduced equations
must be checked in individual cases.
}
\end{rem}

\section{Chapman--Enskog approximation} \label{CEapprox}

Integrating the first equation of \eqref{relax} and
noting that $f(u,v)_\pm=f_*(u_\pm)$, we obtain
\begin{equation}\label{intprof}
\begin{aligned}
f(u,v)&=  f_*(u\pm),\\
A_{21}(u,v)u'+ A_{22}(u,v)v'&=q(u,v).
\end{aligned}
\end{equation}

Taylor expanding the first equation, we obtain
$$
f(u,0)+ f_v(u,0)v + O(v^2)= f_*(u_\pm),
$$
or
\begin{equation}\label{T1}
f_*(u) + f_v(u,0)v + O(v^2)= f_*(u_\pm).
\end{equation}
Taylor expanding the second equation, we obtain
$$
A_{21}(u,0)u' + O(|v||u'|)+ O(|v'|)=\partial_v q(u,0)v + O(|v|^2),
$$
or, inverting $\partial_v q$,
\begin{equation}\label{T2}
v=\partial_v q(u,0)^{-1}A_{21}(u,0) u' + O(|v|^2) +O(|v||u'|) + O(|v'|).
\end{equation}
Substituting \eqref{T2} into \eqref{T1} and rearranging, 
we thus obtain the approximate viscous profile ODE
\begin{equation}\label{approxprof}
b_*(u)u'= f_*(u) -f_*(u_\pm)  + O(v^2) +O(|v||u'|) + O(|v'|).
\end{equation}

Motivated by \eqref{T2}--\eqref{approxprof}, we define an approximate
solution $(\bar u_{CE}, \bar v_{CE})$ of \eqref{intprof} by choosing 
$\bar u_{CE}$  as a solution of 
\begin{equation}
\label{NS}
b_*(\bar u_{CE})\bar u_{CE}' = f_*(\bar u_{CE}) -f_*(u_\pm),
\end{equation}
and $\bar v_{CE}$  as the first approximation given by \eqref{T2} 
\begin{equation}
\label{NSv1}
\begin{aligned}\bar v_{CE}   =    c_* (\bar u_{CE}) 
 \bar u_{CE}'.
\end{aligned}
\end{equation}

\subsubsection{Higher-order correctors}
Further expanding the second equation as
$$
A_{21}(u,0)u' + A_{22}(u,0)v'+ O(|v||u'|+|v||v'|)=\partial_v q(u,0)v + O(|v|^2)
$$
and setting $v=\bar v_{CE} + \tilde v$, $u=\bar u_{CE}$, we obtain
$$
A_{22}(u,0)(\bar v_{CE})'+ O(|v||u'|+|v||v'|)=\partial_v q(u,0)\tilde v + O(|v|^2)
$$
or, inverting $\partial_v q$,
\begin{equation}\label{T2new}
\tilde v=\partial_v q(u,0)^{-1}A_{22}(u,0) u' + O(|v|^2+|v||u'|+|v||v'|). 
\end{equation}
Accordingly, we define 
\be\label{NSv2}
\bar v_{CE,2}=\bar v_{CE}+ 
\partial_v q(\bar u_{CE},0)^{-1}A_{22}(u_{CE},0) u_{CE}'
\ee 
as a second-order corrector for $v$.
Substituting $\bar v_{CE,2}$ into the first equation and discarding
the Taylor remainder as before, we obtain a second-order
 corrector $\bar u_{CE,2}$ for $u$.
We can continue this process of Chapman--Enskog expansion
to all orders to obtain an approximation
\be\label{CEn}
\bar U_{CE}^N:=\bar U_{CE,1}+
\bar U_{CE,2}+ \dots,
\bar U_{CE,N}
\ee
to order $N$, where $\bar U_{CE,1}:=\bar U_{CE}$ is the basic approximant
at the first step.
 
\subsubsection{Existence and decay bounds}
Small amplitude shock profiles   solutions of \eqref{NS}  are constructed 
using the center manifold  analysis of \cite{Pe}
under conditions (i)-(ii) of Assumption \ref{goodred}; see discussion  in 
\cite{MaZ4}.

\begin{prop}\label{NSprofbds} Under Assumptions~\ref{profass} and \ref{goodred}, 
in a neighborhood of 
$(u_0, u_0)$ in $\RR^n \times \RR^n$, 
there is a smooth  manifold $\cS$ of dimension $n$  passing through $(u_0, u_0)$,  such that 
for $(u_-, u_+) \in \cS$ with   amplitude $\eps:=|u_+ -u_-| > 0$ 
sufficiently small, and direction $(u_+-u_-)/\eps $ sufficiently close
to $r(u_0)$,   the zero speed shock profile equation   \eqref{NS} has  a unique (up to translation) 
solution   $\bar u_{CE}$ in $\cU_*$. 
The shock profile is necessarily of {\rm Lax type}: i.e., with
dimensions of the unstable subspace of $df_*(u_-)$
and the stable subspace of $df_*(u_+)$ summing to one plus the
dimension of $u$, that is $n+1$.

Moreover, 
there is  $\theta>0$ and for all $k$ there is $C_k $ independent of $(u_-, u_+) $ and $\eps$,   
such that 
\begin{equation}\label{NSbds}
|\partial_x^k (\bar u_{CE}-u_\pm)|\le C_k \eps^{k+1}e^{-\theta \eps|x|},
\quad x\gtrless 0. 
\end{equation}
and, more generally,
\begin{equation}\label{higherNSbds}
|\partial_x^k (\bar u_{CE,j}|\le C_k \eps^{j+k+1}e^{-\theta \eps|x|},
\quad x\gtrless 0. 
\end{equation}
\end{prop}

We denote by 
 $\cS_+$  the set  of $(u_-, u_+) \in \cS $  with  amplitude $\eps:=|u_+ -u_-| > 0$ 
sufficiently small  and direction $(u_+-u_-)/\eps $ sufficiently close
to $r(u_0)$ such that the profile $\bar u_{CE}$ exists.  
Given $(u_-, u_+) \in \cS_+  $ with associated profile $\bar u_{CE}$, 
we define $\bar v_{CE} $ by \eqref{NSv1} and 
    \begin{equation}
    \label{NSU}
    \bar U_{CE}^N := (\bar u_{CE}^N, \bar v_{CE}^N).
    \end{equation}
    It  is an approximate solution of \eqref{intprof} in the following sense: 

\begin{cor}\label{redbds}
For fixed $u_-$ and amplitude $\eps:=|u_+-u_-|$ sufficiently small,
\begin{equation}\label{eq:resbds}
\begin{aligned}
\cR_u^N&:= f(\bar u_{CE}^N,\bar v_{CE}^N)- f_*(u\pm)
,\\
\cR_v^N&:= g(\bar u_{CE}^N,\bar v_{CE}^N)'-q(\bar u_{CE}^N,\bar v_{CE}^N) 
\end{aligned}
\end{equation}
satisfy
\begin{equation}\label{L2resbds}
\begin{aligned}
|   \D_x^k \cR_u^N  (x) | 
&\le  C_{k,N} \eps^{N+k+4}e^{-\theta \eps|x|} , \\
|   \D_x^k \cR_v^N  (x) | 
&\le  C_{k,N} \eps^{N+k+3}e^{-\theta \eps|x|} , 
\quad x\gtrless 0,\\
\end{aligned}
\end{equation}
where $C_{k,N}$   is  independent of $(u_-, u_+) $ and  $\eps=|u_+ -u_-|$. 
\end{cor}

\begin{proof}
For $N=0$, $k=0$,
bounds \eqref{L2resbds} follow by expansions \eqref{T1} and \eqref{T2}, 
definitions \eqref{NS} and \eqref{NSv1}, and bounds \eqref{NSbds}.  
Bounds for $N$, $k>0$ follow similarly.
\end{proof}

 %

\section{Statement of the main theorem}\label{results}

We are now ready to state the main result.
Define a base state $U_0=(u_0,0)$ and a
neighborhood $\cU=\cU_*\times \cV$.

\begin{theo}\label{main}
Let Assumptions \ref{SS}, \ref{GC}, and  \ref{goodred} hold on the
neighborhood $\cU$ of $U_0$, with $f,A,Q\in C^{\infty}$. 
Then, there are $\eps_0 > 0$  and 
$\delta > 0$ such that for $(u_-, u_+) \in \cS+$ with  amplitude $\eps:=|u_+-u_-| \le \eps_0$,   the standing-wave equation 
\eqref{relax} has a solution   
$\bar U$ in $\cU$, 
 with associated Lax-type 
equilibrium shock $(u_-,u_+)$, satisfying for all $k  $, $N$: 
\begin{equation}\label{finalbds}
\begin{aligned}
\big|\partial_x^k (\bar U- \bar U_{CE}^N)\big|
&\le C_{k,N} \eps^{k+N+2}e^{-\delta  \eps|x|},\\
|\partial_x^k (\bar u-u_\pm)|&\le C_k \eps^{k+1}e^{-\delta \eps|x|},
\quad x\gtrless 0,\\
\big|\partial_x^k (\bar v-v_*(\bar u)\big|
&\le C_k \eps^{k+2}e^{-\delta  \eps|x|},\\
\end{aligned}
\end{equation}  
where $\bar U_{CE}=(\bar u_{CE}, \bar v_{CE})$ is the 
approximating Chapman--Enskog profile defined in \eqref{NSU}, and
$C_k$, $C_{k,N}$ are independent of  $\eps$. 
Moreover, up to translation, this solution is unique
within a ball of radius $c\eps$ about $\bar U_{CE}$ in norm 
$\eps^{-1/2}\|\cdot\|_{L^2}+\eps^{-3/2}\|\D_x \cdot\|_{L^2}+\dots+
 \eps^{-11/2}\|\D_x^5 \cdot\|_{L^2} $, for $c>0$ sufficiently small
and $K$ sufficiently large.
(For comparison, $\bar U_{CE}-U_\pm$ is order $\eps$ in this norm,
by \eqref{finalbds}(ii)--(iii).)
\end{theo}
 
Bounds \eqref{finalbds} show that (i) the behavior of profiles
is indeed well-described by the Navier--Stokes approximation,
and (ii) profiles indeed satisfy the exponential decay rates
required for the proof of spectral stability in \cite{MaZ3}.
From the second observation, we obtain immediately from
the results of \cite{MaZ3} the following stability result.

\begin{cor}[\cite{MaZ3}]\label{MaZcor}
Under the assumptions of Theorem \ref{main}, the resulting
profiles $\bar U$ are spectrally stable for amplitude $\eps$
sufficiently small, in the sense that the linearized operator
$L:= \partial_x A(\bar U) -dQ(\bar U)$ about $\bar U$
has no $L^2$ eigenvalues
$\lambda$ with $\Re \lambda \ge 0$ and $\lambda \ne 0$.
\end{cor}

\begin{proof}
In \cite{MaZ3}, under the same structural conditions assumed here,
it was shown that small-amplitude profiles
of general quasilinear relaxation systems are spectrally stable,
provided that
   $|\bar U'|_{{}_{L^\infty}}\leq C|U_+-U_-|^2$,
   $|\bar U''(x)|\leq C|U_+-U_-|\,|\bar U'(x)|$,
and
\begin{equation}\label{dirder}
   \Big|\frac{\bar U'}{|\bar U'|} +\sgn (\eta) R_0\Big|\leq C\,|U_+-U_-|,
\qquad
R_0:=\begin{pmatrix}r(u_0)\\ dv_*(U_0)r(u_0)\end{pmatrix},
\end{equation}
where $r(u_0)$ as defined in Theorem \ref{main}
is the eigenvector of $df_*$ at base point $U_0$ in
the principal direction of the shock.
These conditions are readily verified using \eqref{finalbds}.
\end{proof}

The remainder of the paper is devoted to the proof of Theorem \ref{main}.

\section{Outline of the proof}\label{outline}

\subsection{Linear and nonlinear perturbation equations}\label{pertsec}
 
Defining the perturbation variable $U:= \bar U- \bar U_{CE}^N$,
where $\bar U_{CE}^N$ is as in \eqref{CEn},
we obtain from \eqref{intprof} the nonlinear perturbation equations
$\Phi^\eps(U)=0$, where
\begin{equation}\label{intpert}
\Phi^\eps(U):=
\begin{pmatrix}
 f_1(\bar U^\eps_{CE}+U)-f_*(U_-) \\
(A_{21}(\bar U^\eps_{CE} +U) (\bar u^\eps_{CE} +u)'
+(A_{22}(\bar U^\eps_{CE} +U) (\bar v^\eps_{CE} +v)'
-q(\bar U^\eps_{CE}+U)
\end{pmatrix}.
\end{equation}
Formally linearizing $\Phi^\eps$ about an approximate solution $\tilde U$,
we obtain 
\begin{equation}\label{Phi'}
(\Phi^\eps)'(\tilde U)U=
\begin{pmatrix}
A_{11} u +A_{12} v \\
A_{21} u'+  A_{22} v'
-Q_{22}v -bU
\end{pmatrix},
\end{equation}
where 
\begin{equation}\label{Adef}
A=df(\bar U^\eps_{CE}+\tilde U),
\quad
Q_{22} = \D_v q (\bar U^\eps_{CE}+\tilde U),
\end{equation}
and
\begin{equation}\label{bdef}
bU = \big( d(A_{21},A_{22})(\bar U^\eps_{CE}+\tilde U)U\big) 
(\bar U^\eps_{CE}+\tilde U)'.
\end{equation}

The associated linearized equation for a given forcing term
$F$ is 
\begin{equation}\label{linpert}
(\Phi^\eps)'(\tilde U) U= F=\begin{pmatrix} f\\g\end{pmatrix}.
\end{equation}

We have also
\begin{equation}\label{defPhi''}
(\Phi^\eps)''(\tilde U)(U,\hat U)=
\begin{pmatrix}
N_1(\tilde U)(U,\hat U)\\
N_2(\tilde U)(U,\hat U)'+N_3(\tilde U)(U,\hat U)
\end{pmatrix},
\end{equation}
where $N_j(\tilde U)$ are quadratic forms depending smoothly on
$\tilde U$.

\subsection{Functional analytic setting}\label{norms}
The coefficients  and the error term $\cR$ are smooth functions of 
$\bar U_{CE'}$ and its derivative, so behave like smooth functions of 
$ \eps x$. Thus, it is natural to solve the equations in spaces which reflect 
this scaling. We do not introduce explicitly the change of variables
$\tilde x = \eps x$, but introduce norms which correspond to the usual $H^s$ norms 
in the $\tilde x $ variable : 
\begin{equation} \label{defnorm}
\|f \|_{H^s_\eps} =  
\eps^{\mez}  \|f \|_{L^2}+
\eps^{-\mez }\|\partial_x f\|_{L^2}+ \dots + 
\eps^{\mez-s}\|\partial_x^s f\|_{L^2}.
\end{equation}
We also introduce weighted spaces and norms, which encounter for the exponential 
decay of the source and solution: introduce the notations.  
\begin{equation}
\label{modx}
<x>:= (x^2+1)^{1/2}
\end{equation}
For  $\delta \ge 0$ (sufficiently small), we denote by $H^s_{\eps, \delta}$ the space of 
functions $f$ such that   $ e^{\delta  \eps <x>} f \in H^s$ equipped with the norm
\begin{equation}
\label{defwnorm}
\|f \|_{H^s_{\eps, \delta} } =   \eps^{\mez} \sum_{k \le s} \eps^{-k}  \|e^{\delta \eps <x>} \D_x^k f \|_{L^2}.
\end{equation}
 Note that for $\delta \le 1$, this norm is equivalent, with constants independent of $\eps$ and $\delta$, 
 to the norm
 $$
\|e^{\delta \eps <x>}  f \|_{H^s_\eps} . 
 $$
For fixed $\delta$, introduce spaces $E_s:=H^s_{\eps,\delta}$
with norm $\|\cdot\|_s=\|\cdot\|_{H^s_{\eps,\delta}}$ and
$F_s:=\bp H^{s+1}_{\eps,\delta}\\ H^s_{\eps,\delta} \ep$
with norm $|\bp f\\g\ep|_s=\|f\|_{H^{s+1}_{\eps,\delta}}
+\|g\|_{H^{s}_{\eps,\delta}}$.

\subsection{Nash Moser iteration scheme}\label{nmscheme}

\begin{lem}\label{WKBapprox}
$|\Phi(0)|_{H^s_{\delta, \eps}}\le C\eps^{N+2}$
for all $0\le s\le \bar s$, some $C>0$.
\end{lem}

\begin{proof}
Immediate from \eqref{L2resbds} and \eqref{defnorm}.
\end{proof}

\begin{lem}\label{tamederivs}
$\Phi^\eps$ is Frechet differentiable 
from $H^{s+1}_{\eps,\delta}\to H^{s}_{\eps,\delta} $, 
for all $s\ge 0$, $\eps>0$, $\delta\ge 0$, and, 
for $s_0\ge 1$, all $s$ such that 
 $s_0 + 1 \leq s + 1 \leq \bar s,$ 
and all $U,V,W \in H^{s+1}_{\eps,\delta},$
\begin{equation}\label{tamePhi}
  | \Phi^\e(U)|_{s} \leq C_0 (1 + | U |_{{H^{s+1}_{\eps,\delta}}} + |U|_{H^{s_0 + 1}_{\eps,\delta}}|U|_{H^s_{\eps,\delta}}), 
\end{equation}
\begin{equation}\label{tamePhi'}
  | (\Phi^\e)'(U) \cdot V|_{s}  \leq  
C_0 (|V|_{H^{s+1}_{\eps,\delta}} + |V|_{H^{s_0+1}_{\eps,\delta}}
|U|_{H^{s+1}_{\eps,\delta}}), 
\end{equation}
and
  \begin{equation} \begin{aligned} \label{Phi''} 
| (\Phi^\e)''(U) \cdot (V, W)|_{s} \leq C_0 & \big(|V|_{H^{s_0 + 1}_{\eps,\delta}} |W|_{H^{s + 1}_{\eps,\delta}} + 
|V|_{H^{s + 1}_{\eps,\delta}} |W|_{H^{s_0 + 1}_{\eps,\delta}} \\ 
& + |U|_{H^{s + 1}_{\eps,\delta}}|V|_{H^{s_0 + 1}_{\eps,\delta}} 
|W|_{H^{s_0 + 1}_{\eps,\delta}}\big),
  \end{aligned}\end{equation}
where $C$ is uniformly bounded for $|U|_{H^{s_0+1}_{\eps,\delta}}\le C$,
for any fixed value of $\delta$.
\end{lem}

\begin{proof}
Standard, using Moser's inequality, definition \eqref{Phi'}, 
the fact that $|\cdot|_{H^s_{\eps,\delta}}$ is a fixed weighted
norm in coordinates $\tilde x=\eps x$, and working in $\tilde x$
coordinates, with $\partial_x=\eps\partial_{\tilde x}$.
\end{proof}

\begin{prop}\label{invprop}
Under the assumptions of Theorem \ref{main},  
there are $\eps_0 > 0 $ and $\delta > 0$ such that
for all  $\eps \in ]0, \eps_0]$, 
for $\eps \in ]0, \eps_0]$, 
$\delta \in [0, \delta_0]$, 
equation \eqref{linpert} has a solution operator $\Psi^\eps(\tilde U)$ (i.e.,
there exists a formal right inverse for $(\Phi^\eps)'(\tilde U)$),
such that, for all $s$ such that
 $ s_0 + 2  \leq s + 1 \leq \bar s$, 
$s_0=3$,
$F=\begin{pmatrix} f\\g\end{pmatrix} \in F_s$,
and $U \in H^{s + r}_{\eps,\delta}$ such that
\begin{equation} \label{cond1first}
 | \tilde U |_{H^{s_0 + 2}_{\eps,\delta}} \le C\e,
 \end{equation}
there holds the estimate
 \begin{equation}\label{invbdHs}
\begin{aligned}
\big\|\csLd  F \big\|_{H^s_{\eps, \delta} }&\le 
C \eps^{-1}\big( 
\big\| \tilde U \|_{H^{s+1}_{\eps, \delta} }
\big| F |_{{s_0+2}}
+
\big\| F \|_{{s+1}}\big)\\
&=
C \eps^{-1}\big( 
\big\| \tilde U \|_{H^{s+1}_{\eps, \delta} }
(\big\| f \|_{H^{s_0+3}_{\eps, \delta} }
+
\big\| g \|_{H^{s_0+2}_{\eps, \delta} })
+
(\big\| F \|_{H^{s+2}_{\eps, \delta} }
+\big\| g \|_{H^{s+1}_{\eps, \delta} }  \big)),
\\
\end{aligned}
\end{equation}
where $C=C(|\tilde U|_{H^{s_0+2}_{\eps, \delta}})$ 
  is a non-decreasing function of $|\tilde U|_{H^{s_0+2}_{\eps, \delta}}$.
\end{prop}

The proof of this proposition, carried out in 
Sections \ref{energy}--\ref{existence} is essentially identical to that
of the corresponding proposition (Prop. 5.2) of \cite{MZ2}
in the semilinear case.
Once it is established, existence and uniqueness follow by 
the abstract Nash--Moser theorems developed in
\cite{TZ}, reproduced for completeness in Appendix \ref{NM}.

\begin{proof}[Proof of Theorem \ref{main} (Existence)]
The profiles $\bar U_{CE}^N$ exist if $\eps$ is small enough. 
Comparing, we find that
Lemma \ref{tamederivs},  Proposition \ref{invprop},
and Lemma \ref{WKBapprox} verify, respectively, Assumptions \ref{ass1},
\ref{ass2}, and \ref{ass-wkb} of Appendix \ref{NM},
with $s_0=3$, $\gamma_0=0$, $\g=1$, $k=N+2$, $m=r=1$, $r'=0$, 
and arbitrary $\bar s$. 
Taking $\bar s$ sufficiently large, and applying
the Nash Moser Theorem \ref{th1} of Appendix \ref{NM}, we thus
obtain existence of a solution $U^\eps$ of \eqref{intpert}
with $|U^\eps|_{H^{s+1}_{\eps, \delta}}\le C\eps^{N+1}$.
Defining $\bar U^\eps:=\bar U_{CE}^N+U^\eps$, and
noting by Sobelev embedding that $|h|_{H^{s+1}_{\eps, \delta}}$
controls $| e^{\delta \eps |x|}h|_{L^\infty}$, we obtain the result.
\end{proof}

\begin{proof}[Proof of Theorem \ref{main} (Uniqueness)]
Applying Theorem \ref{th2}
for $s_0=3$, $\gamma_0=0$, $\g=1$, $k=3$, $m=r=1$, $r'=0$, 
we obtain uniqueness in a ball of radius $c\eps$ in 
$H^{4}_{\eps,0}$, 
$c>0$ sufficiently small,
under the additional phase condition \eqref{phasecond}.
We obtain unconditional uniqueness from this weaker version
by the observation that phase condition
\eqref{phasecond} may be achieved for any solution $\bar U=\bar U_{CE}+U$
with 
$$
\|U'\|_{L^\infty}\le c \eps^{2}<< \bar U_{CE}'(0)\sim \eps^2
$$
by translation in $x$, yielding
$\bar U_a(x):=\bar U(x+a)= \bar U_{CE}(x)+ U_a(x)$
with 
$$
U_a(x):= \bar U_{CE}(x+ a)-\bar U_{CE}(x)+ U(x+a) 
$$
so that, defining $\phi:=\bar U'/|\bar U'|$, 
we have
$\partial_a \langle \phi, U_a\rangle \sim 
\langle \phi, \bar U_{CE}' + U' \rangle
=\langle \phi, (1+o(1))\bar U' + U' \rangle
= (1+o(1)) |\bar U'|\sim \eps^2$
and so (by the Implicit Function Theorem applied to $h(a):=\eps^{-2}
\langle \phi, U_a\rangle$, together with the
fact that $\langle \phi, U_0\rangle = o(\eps)$ and
that $\langle \phi, \bar U_{NS}'\rangle\sim |\bar U_{NS}'|\sim \eps^{2}$)
the inner product $\langle \phi, U_a\rangle$, hence also
$\Pi U_a$
may be set to zero by appropriate choice of $a=o(\eps^{-1})$ leaving
$U_a$ in the same $o(\eps)$ neighborhood, by the computation
$U_a-U_0\sim \partial_a U \cdot a\sim o(\eps^{-1})\eps^2$.
\end{proof}

It remains to prove existence of the linearized solution
operator and the linearized bounds \eqref{invbdHs}, which
tasks will be the work of the rest of the paper.
We concentrate first on estimates, and prove the existence next, using 
a viscosity method.

\section{Internal and high frequency estimates}
 \label{energy}

We begin by establishing a priori estimates on solutions
of the equation \eqref{linpert}
This will be done in two stages.
In the first stage, carried out in
this section, we establish energy estimates
showing that ``microscopic'', or ``internal'', variables consisting
of $v $ and derivatives of $(u, v)$ are controlled by  and small 
with respect to the ``macroscopic'', or ``fluid'' variable, $u$.
In the second stage, carried out in Section \ref{linCEestimates}, we
estimate the macroscopic variable $u$ by Chapman--Enskog approximation
combined with finite-dimensional ODE techniques such as have been
used in the study of fluid-dynamical shocks \cite{MZ1,MaZ5,PZ,Z1}.

\subsection{The basic $H^1$ estimate} 
  
  We consider the equation
  \begin{equation}
  \label{inteqs6}
\begin{pmatrix} A_{11} u + A_{12} v  
  \\
   A_{21} u' + A_{22} v'  +bU -   Q_{22} v \end{pmatrix} = 
  \begin{pmatrix} f \\ g \end{pmatrix}  
  \end{equation}
   and its differentiated form: 
\begin{equation}\label{apriorieq}
(AU'- Q+b)U=
\begin{pmatrix} f'\\g \end{pmatrix},
\end{equation}
where $b=\tilde b(\bar U_{CE}^N)'$, and $A$, $Q$, $\tilde b$
are smooth functions of $\bar U_{CE}+\tilde U$, with
$\|\tilde U\|_{4}$,
$\|\bar U_{CE}^N\|_{s+1}$ both order $\eps$ (the first by assumption,
the second by estimates \eqref{higherNSbds}).
We shall freely use below the resulting coefficient bounds
\be\label{coeffests}
|\partial_x^{k+1} A|, \, 
|\partial_x^{k+1} Q|,\,
|\partial_x^{k+1} K|,\
\le C\eps^{2+k},
\quad
|\partial_x^k b|\le C\eps^{2+k}
\ee
for $0\le k\le 3$
and
\be\label{ecoeffests}
|\partial_x^{j+1} A|_{L^2}, \, 
|\partial_x^{j+1} Q|_{L^2}, \, 
|\partial_x^{j+1} K|_{L^2}
\le C\eps^{j+1/2}(\eps + \|\tilde U\|_{s+1}),
\quad
|\partial_x^{j} b|\le C\eps^{j+1/2}(\eps + \|\tilde U\|_{s+1})
\ee
for $0\le j\le s$, where $K$ is the Kawashima multiplier (a smooth
function of $A$ ).
The internal variables are $U' = (u', v')$ and $v$.


\begin{prop}\label{energypropL2}
Under the assumptions of Theorem \ref{main}, there  are  constants 
$C$, $\eps_0 > 0$ and $\delta_0 > 0$ such that for  $0 < \eps \le \eps_0$ and 
$0 \le \delta \le \delta_0$, 
$f \in H^{2}_{\eps, \delta} $, $g \in H^{1}_{\eps, \delta} $ 
and      $U= (u,v)\in H^1_{\eps, \delta}$ satisfying \eqref{inteqs6},
one has
  \begin{equation}\label{invbd}
\big\| U'   \big\|_{L^2_{\eps, \delta} }   + \big\| v   \big\|_{L^2_{\eps, \delta} }  \le 
C    \big( \big\| (f, f', f'', g, g') \|_{L^2_{\eps, \delta} }
 + \eps  \big\| u  \big\|_{L^2_{\eps, \delta} }  \big).
\end{equation}

\end{prop}

Multiplying by symmetrizer $S$ (block-diagonal, by assumption
\eqref{zerov}), 
we obtain an ODE
\begin{equation} \label{simplfeqtilde}
\widetilde  A \widetilde U'-\widetilde{ Q}\, \widetilde U 
+ \widetilde  C \widetilde U= \widetilde F,
\end{equation}
where 
\begin{equation}
\label{hateq}
\widetilde A=SA,\quad  \widetilde { Q}=SQ=
\begin{pmatrix}
0 & 0 \\
0 &\widetilde{Q}_{22} \end{pmatrix},
\end{equation}
with $ \Re \widetilde{Q}_{22}$ negative definite,
$\widetilde F=SF$,
and
\begin{equation}\label{Cest}
\widetilde  C=
O(\bar u_{CE}') \widehat C =O(\eps^2)\widehat C
\end{equation}
comprising commutator terms and $-S_{22}bU$.
\medbreak

We first prove the estimate \eqref{invbd} for $\delta = 0$. 
Dropping hats and tildes, the  ODE reads
\begin{equation}
\label{simplfeq}
AU' -   Q  U+\eps^2 CU = F, 
\quad
Q =
\begin{pmatrix}
0 & 0 \\
0 &   Q_{22}
\end{pmatrix},
\end{equation}
$A$ symmetric and $\Re Q_{22}$ negative definite.
Likewise, the genuine coupling condition still holds,
which, by the results of \cite{K}, is equivalent to the 
{\it Kawashima condition}, and 
there is a smooth $  K = \widetilde K(\bar u_{CE})= - \widetilde K^* $ such that 
$
\Re  (  K   A  -    S  {Q}) 
$ is definite positive. Therefore, there is  $c > 0$ such that for all 
$\eps \le \eps_0$ and $x \in \RR$: 
\begin{equation}
\label{infbds}
\tilde  q \le   - c \Id, \qquad 
\Re (  K   A -   S  {Q}) \ge c\Id.
\end{equation}

\begin{lem}
\label{lem62}
There is a constant $C$ such that for $\eps$ sufficiently small, 
$f \in H^2$, $\tilde U\in H^2$, $g \in H^1$, and  $U\in H^1$  
satisfying \eqref{simplfeq}, with $\|\tilde U\|_{2}\le C\eps$,
one has 
\begin{equation}\label{sharph1eq}
\| U'  \|_{L^2} + \| v \|_{L^2}  \le C \big(     
\|  f  \|_{H^2}    + \| g  \|_{H^1}   
+ \eps \| u \|_{L^2} \big) . 
\end{equation}
\end{lem}

\begin{proof}
 
 Introduce the symmetrizer
\begin{equation} 
\label{def615}
\cS = \D_x^2    + \D_x \circ  K   -  \lambda   .  
\end{equation}
One has 
$$
\begin{aligned}
& \Re \D_x^2    \circ (A\D_x  -      Q)  =    \mez \D_x \circ  A' \circ \D_x  -     \D_x   \circ   Q \circ    \D_x  -  \Re \D_x  \circ  Q'   
\\
& \Re \D_x \circ K  (A\D_x -  Q)  =  \D_x \circ  \Re KA \circ  \D_x  -   \re \D_x \circ K  Q 
\\
&  \Re   (   A\D_x -    Q )    = \mez   A'   -    Q .
\end{aligned}
$$
Thus 
$$
\begin{aligned}
\Re \cS  \circ ( A\D_x -    Q)  =  & \D_x \circ (\Re AK  -    Q) \circ \D_x   +   \lambda    Q  
\\ 
& + \mez \D_x  \circ A' \circ \D_x -  \mez  \lambda   A'   
-   \Re \D_x \circ  Q' -   \Re \D_x \circ K  Q. 
\end{aligned}
$$
Therefore, for  $U \in H^2 (\RR)$,  \eqref{infbds} implies that 
$$
\begin{aligned}
   \Re ( \cS   F , U)_{L^2}   \ge & \ c \| \D_x U \|^2_{L^2} + 
\lambda c  \|   v \|^2_{L^2} 
\\ &- \mez \| ( A)'  \|_{L^\infty}  \big( \| \D_x U \|^2_{L^2} + \lambda  \|  U \|^2_{L^2}\big) 
\\
& - \| (Q)' \|_{L^\infty} \|  U \|_{L^2} \| \D_x U \|_{L^2}  - \| K   \|_{L^\infty} \|\D_x  U \|_{L^2} 
\|  q v  \|_{L^2} \\
&-\eps^2(|C|_{L^\infty}|U|_{H^1}^2
+|C'|_{L^2}|U|_{L}^\infty). 
\end{aligned}
$$
Taking
 $$
\lambda =  \frac{2}{c} \| K  \|^2_{L^\infty} \| q \|_{L^\infty} ,
$$
and using that 
\begin{equation}
\label{commest}
  \| ( A)'  \|_{L^\infty} + \| (Q)'  \|_{L^\infty} = O( \eps^2),
\quad  
  \| ( C)'  \|_{L^2} \sim \| (A)''  \|_{L^2}\sim 
\eps^{3/2}(\eps + \|\tilde U\|_{2})=O(\eps^{5/2})
\end{equation}
and $|U|_{L^\infty}\le 
\|U\|_{1}=\eps^{-1/2}\|U\|_{L^2}+\eps^{1/2} \|U\|_{H^1}$,
yields
$$
\| U' \|^2_{L^2} + \| v \|_{L^2}^2 \lesssim     \Re ( \cS   F , U)_{L^2}   + 
\eps^2 \big( \| U \|^2_{L^2} +  \| U' \|^2_{L^2} \big).  
$$

In the opposite direction, 
$$
\begin{aligned}
   \Re ( \cS F , U)_{L^2}   \le      & \| \D_x  U \|_{L^2} \big( 
 \| \D_x (F) \|_{L^2}  + \| K \|_{L^\infty} \| F \|_{L^2} \big) 
\\
& + \lambda \big( \|  ( u) '  \|_{L^2}   \|  f \|_{L^2}  
 +   \|  v \|_{L^2}   \| g  \|_{L^2} \big).  
\end{aligned}
$$
Using again    
that the derivatives of the coefficients are $O(\eps^2)$,  this implies  that 
\begin{equation*} 
\begin{aligned}
 \Re ( \cS F , U)_{L^2}  \lesssim  & 
 \big(   
\|  f  \|_{H^2}   +    \| g \|_{H^1} \big) \| U' \|_{L^2}  
\\ 
 &+      \eps^2 \| f \|_{L^2} \| u \|_{L^2}  + \| g \|_{L^2} \| v \|_{L^2},
\end{aligned}
\end{equation*}
The estimate \eqref{sharph1eq}  follows provided that $\eps$  is small enough. 

This proves the lemma under the additional assumption that $U \in H^2$. 
When $U\in H^1$, the estimates follows using Friedrichs mollifiers. 
\end{proof}

\begin{proof}[Proof of Proposition~\ref{energypropL2}]
This follows similarly as in the proof of Lemma \ref{lem62},
making the change of variables $U\to e^{\delta \eps |x|}U$
and absorbing commutators.  See the proof of Proposition 6.1,
\cite{MZ2}.
\end{proof}

\subsection{Higher order estimates} 
\begin{prop}\label{estHs}  There are constants $C$, $\eps_0 > 0$, $\delta_0 > 0$ and for 
all $k \ge 2$, there is  $C_k$, such that 
 $0< \eps \le \eps_0$, $\delta \le \delta_0$,  
 $U \in H^s_{\eps, \delta}$, 
 $\tilde U \in H^{s+1}_{\eps, \delta}$,
  $f \in H^{s+1}_{\eps, \delta}$ and $g \in H^{s}_{\eps, \delta}$ 
  satisfying \eqref{simplfeq}, 
with $\|\tilde U\|_{H^2_{\eps, \delta}}\le C\eps$, there holds
\begin{equation}\label{sharph2eq}
\begin{aligned}
  \| \D_x^k U' \|_{L^2_{\eps, \delta}} + 
&  \| \partial^k_x v \|_{L^2_{\eps, \delta}}  
  \le  C  
  \| \partial_x^k (f, f', f'', g, g')  \|_{L^2_{\eps, \delta} }     \\
  &  +    \eps^k  C_k \big( 
 \| U' \|_{H^{k-1}_{\eps, \delta} } + \eps   \| v \|_{H^{k-1}_{\eps, \delta} }  +
  \eps \| u \|_{L^2_{\eps, \delta}}\big)  \\
&\quad +
C_k\eps^{k+1}\|\tilde U\|_{H^{k+2}_{\eps,\delta}}
(  \|v \|_{H^{1}_{\eps, \delta}} +
\eps\|U \|_{H^{2}_{\eps, \delta}} ).
\end{aligned}
\end{equation}
\end{prop}

\begin{proof} 
Differentiating \eqref{inteqs6} $k$ times, yields
\begin{equation}\label{Dapriorieq}
A\partial_x^{k}U - Q\partial_x^{k}U =
\begin{pmatrix} \partial^k_x f'    \\ \partial^{k}_xg +   r_k   \end{pmatrix},
\end{equation}
where 
$$
r_k  =  - 
\partial_x^{k-1} \big((\partial_x  Q_{22})  v\big)  
- \partial^{k-1}_x \big( (\partial_x A) \,  \D_{x} U\big)
- \partial^{k-1}_x \big(  (\partial_x C) \,  U\big).
$$
The   $H^1$ estimate yields
$$
\begin{aligned}
\| \D_x^k U' \|_{L^2_{\eps, \delta}} + 
 \| \partial^k_x  v   \|_{L^2_{\eps, \delta}}   \le  
  C  \big( & \| \partial_x^k (f, f', f'', g, g')   \|_{L^2_{\eps, \delta}}     \\
     + \eps  \| \partial_x^k  u \|_{L^2_{\eps, \delta}}  
 & +   \| \partial_x r_k \|_{L^2_{\eps, \delta}} 
+   \| r_k  \|_{L^2_{\eps, \delta}} \big) , 
\end{aligned}
$$
for $0 \le k \le s$, with  $r_0 = 0$ when $k = 0$. 

Using Moser's inequality together with 
\eqref{coeffests} and \eqref{ecoeffests}, we may
estimate
$$
\begin{aligned}
   \| r_k  \|_{L^2_{\eps, \delta}} &\le C_k
( |\partial_x Q|_{L^\infty}\|\partial_x^{k-1} v\|_{L^2} +
|\partial_x^k Q|_{L^2}\|v\|_{L^\infty} \\
&\quad +
|\partial_x A|_{L^\infty}\|\partial_x^k U\|_{L^2} +
|\partial_x^k A|_{L^2}\|\partial_x U\|_{L^\infty} \\
&\quad +
|\partial_x C|_{L^\infty}\|\partial_x^{k-1} U\|_{L^2} +
|\partial_x^k C|_{L^2}\| U\|_{L^\infty} 
)
\\
&\le
C_k( \eps^{k+1}
\|v \|_{H^{k-1}_{\eps, \delta}} +
\eps^{k+2}\|U \|_{H^{k-1}_{\eps, \delta}} +
\eps^2 \|\partial_x^k U \|_{L^{2}_{\eps, \delta}}
)\\
&\quad +
C_k( \eps^{k+1}
\|\tilde U\|_{H^k_{\eps,\delta}}
\|v \|_{H^{1}_{\eps, \delta}} +
\eps^{k+2}\|\tilde U\|_{H^k_{\eps,\delta}}\|U \|_{H^{2}_{\eps, \delta}} +
\eps^{k+2}\|\tilde U\|_{H^{k+1}_{\eps,\delta}}\|U \|_{H^{1}_{\eps, \delta}} 
),
\end{aligned}
$$
obtaining the result by absorbing (smaller) 
highest-order terms from $\|\partial_x r_k\|_{L^2_{\eps,\delta}}$ 
on the left-hand side.

\end{proof}


\section{Linearized Chapman--Enskog estimate} \label{linCEestimates}

\subsection{The approximate equations}

It remains only to estimate $\|u\|_{L^2_{\eps, \delta}}$      in order to close the estimates
and establish \eqref{invbd}.
To this end, we work with the first equation  in 
\eqref{inteqs6}
and  estimate it by comparison with the Chapman-Enskog 
approximation (see the computations Section~\ref{CEapprox}). 
  
From the second equation
$$
A_{21}u'+A_{22}v' -g= dq_v v,
$$
we find 
\begin{equation}
\label{T2s7}
v=\partial_v q^{-1}
\Big((A_{21} + A_{22}\partial_v dv_* (\bar u_{CE}))u ' +A_{22} v  ' -g
\Big). 
\end{equation}
Introducing $v$ in the first equation, yields 
$$
(A_{11} + A_{12} d v_* (\bar u_{CE} ) ) u  + A_{12} v =  f,  
$$ 
 thus
$$
(A_{11} + A_{12} dv_* (\bar u_{CE}) ) u' =  f' - A_{12} v' -
d^2 v_* (\bar u_{CE}) (\bar u'_{CE}, u) . 
$$
 Therefore, \eqref{T2s7} can be modified to 
 \begin{equation}
 \label{T2bs7}
 v  =  c_* (\bar u_{CE}) u'   +   r 
\end{equation}
 with 
 $$
 \begin{aligned}
 r = d^{-1}_vq (\bar u_{CE}, &v_* (\bar u_{CE})) \Big(  A_{22}(v)' -g
 \\
& + dv_* (\bar u_{CE}) \big(  f' - A_{12} v' -
d^2 v_* (\bar u_{CE}) (\bar u'_{CE}, u)\big) \Big) . 
 \end{aligned}
 $$
This implies that $u$ satisfies the linearized profile equation
\begin{equation}\label{intpertu}
\begin{aligned}
\bar b_* u'- \bar {df}_* u  =   A_{12} r  - f 
\end{aligned}
\end{equation}
where $\bar b_*=b_*(\bar u_{CE})$ 
and $\bar {df}_{*} := df_*(\bar u_{CE}) = A_{11} + A_{12} dv_* (\bar u_{CE})$.


\subsection{$L^2$ estimates and proof   of the main estimates}

\begin{prop}\label{uprop}
For $\|\tilde U\|_{4}\le C\eps$,
the operator $( \bar b_*\partial_x -\bar{df}_*)(\tilde U)$ has
a right inverse $(b_*\partial_x -df^*)^{\dagger}$ 
\begin{equation}\label{rightinv}
\|(\bar b_*\partial_x -\bar {df}_*)^{\dagger}h\|_{L^2_{\eps, \delta}} \le 
C\eps^{-1}\|h\|_{L^2_{\eps, \delta}},
\end{equation}
uniquely specified by the property that the solution 
$u = (b_*\partial_x -df^*)^{\dagger} h$  satisfies 
\begin{equation}\label{phase}  
\ell_\eps  \cdot u(0) =0. 
\end{equation}
for certain unit vector $\ell_\eps$. 
\end{prop}

\begin{proof}
Standard asymptotic ODE techniques, using the gap and reduction
lemmas of \cite{MZ1,MaZ3,PZ}, where the assumption
$\|\tilde U\|_{H^4_{\eps,\delta}}\le C\eps$ gives the
needed control on coefficients; 
see the proof of Proposition 7.1, \cite{MZ2}.
\end{proof}
 
 \begin{prop}
 \label{prop72}
There are  constants $C$,  $\eps_0 > 0 $ and  $\delta_0 > 0$   
such that 
for $\eps \in ]0, \eps_0]$, $\delta \in [0, \delta_0]$,       
$f \in H^{3}_{\eps, \delta} $, $g \in H^{2}_{\eps, \delta} $ and 
$U \in H^2_{\eps, \delta}$ satisfying  
  \eqref{linpert} and \eqref{phase}  
  \begin{equation}
  \label{invbdH2s7}
\big\| U \big\|_{H^2_{\eps, \delta} }\le 
C\eps^{-1}\big( \big\| f \|_{H^{3}_{\eps, \delta} }
+ \big\|g  \big\|_{H^2_{\eps, \delta} }\big).
\end{equation}
\end{prop}

\begin{proof}
Going back now to \eqref{intpertu}, $u$ satisfies 
$$
\begin{aligned}
\bar b_* u'- \bar {df}_* u &=  O(|v'|+ |g| + |f'| + \eps^2 | u |  )  -   f,
\end{aligned}
$$
   If in addition  $u$ satisfies the condition \eqref{phase}
then  
\begin{equation}
\label{temp2}
\|u\|_{L^2_{\eps, \delta}}\le C   \eps^{-1} 
( \|v'\|_{L^2_{\eps, \delta} }
+ \|(f, f',g)\|_{L^2_{\eps, \delta} }     + \eps^2  \| u \|_{L^2_{\eps, \delta}} \big) . 
\end{equation}

By  Proposition~\ref{energypropL2} and Proposition~\ref{estHs} for $k = 1$, we have 
  \begin{equation} 
  \label{est77}
\big\| U'   \big\|_{L^2_{\eps, \delta} }   + \big\| v   \big\|_{L^2_{\eps, \delta} }  \le 
C    \big( \big\| (f, f', f'', g, g') \|_{L^2_{\eps, \delta} }
 + \eps  \big\| u  \big\|_{L^2_{\eps, \delta} }  \big).
\end{equation}
\begin{equation}
\label{est78}
\begin{aligned}
  \|   U'' \|_{L^2_{\eps, \delta}} + &
   \big\| v '   \big\|_{L^2_{\eps, \delta} }  \le 
   \\
& C    \big( \big\| (f', f'', f''',  g',   g'') \|_{L^2_{\eps, \delta} }
 + \eps  \big\| U'  \big\|_{L^2_{\eps, \delta} }   + 
 \eps^2   \big\| u  \big\|_{L^2_{\eps, \delta} } \big).
\end{aligned}
 \end{equation}
Combining these estimates,  
this implies 
 \begin{equation*} 
 \begin{aligned}
  \big\| v '   \big\|_{L^2_{\eps, \delta} }  & \le    
C      \big( \big\| (f', f'', f''',  g',   g'') \|_{L^2_{\eps, \delta} } + \eps \big\| (f, f', f'', g, g') \|_{L^2_{\eps, \delta} }
 + \eps^2  \big\| u  \big\|_{L^2_{\eps, \delta} }  \big)\\
& \le    
C      \big(   \eps \big\| (f, f', f'', g, g') \|_{H^1_{\eps, \delta} }
 + \eps^2  \big\| u  \big\|_{L^2_{\eps, \delta} }  \big).
 \end{aligned}
\end{equation*}
Substituting in \eqref{temp2}, yields 
$$ 
\eps \|u\|_{L^2_{\eps, \delta}} \le C \big(  \|(f, f',g )\|_{L^2_{\eps, \delta}} + 
 \eps \|(f, f',f'',g, g' )\|_{H^1_{\eps, \delta}}    
+ \eps^2 \|  u \|_{L^2_{\eps, \delta}} \big). 
$$
Hence for $\eps $ small,  
\begin{equation}\label{temp3}
\eps \|u\|_{L^2_{\eps, \delta}} \le C \big(  \|(f, f',g )\|_{L^2_{\eps, \delta}} + 
 \eps \|(f, f',f'',g, g' )\|_{H^1_{\eps, \delta}}     \big). 
\end{equation}
 
Plugging this estimate in \eqref{est77} 
 \begin{equation} 
  \label{est711}
\big\| U'   \big\|_{L^2_{\eps, \delta} }   + \big\| v   \big\|_{L^2_{\eps, \delta} } 
+  \eps  \big\| u  \big\|_{L^2_{\eps, \delta} }  \le 
C     \big\| (f, f', f'', g, g') \|_{H^1_{\eps, \delta} }.
\end{equation}
Hence, with \eqref{est78}, one has 
\begin{equation}
\label{est712}
\begin{aligned}
  \|   U'' \|_{L^2_{\eps, \delta}} + &
   \big\| v '   \big\|_{L^2_{\eps, \delta} }  \le 
   \\
& C    \big( \big\| (f', f'', f''',  g',   g'') \|_{L^2_{\eps, \delta} }
 + \eps   \big\| (f, f', f'', g, g') \|_{H^1_{\eps, \delta} }\big).
\end{aligned}
 \end{equation}
Therefore, 
 \begin{equation}
\label{est713}
 \big\| U'   \big\|_{H^1_{\eps, \delta} }          
+  \big\| v   \big\|_{L^2_{\eps, \delta} } 
+  \eps  \big\| u   \big\|_{L^2_{\eps, \delta} }  \le 
  C     \big\| f, f', f'', g, g'  \big\|_{H^1_{\eps, \delta} }  
\end{equation}
The left hand side dominates 
$$
  \big\| U'   \big\|_{H^1_{\eps, \delta} }  + \eps  \big\| U'   \big\|_{L^2_{\eps, \delta} } 
  =  \eps \big\| U'   \big\|_{H^2_{\eps, \delta} } 
  $$
and the right hand side is smaller than or equal to  
$  \big\|  f  \big\|_{H^2_{\eps, \delta} } +  \big\| g   \big\|_{H^1_{\eps, \delta} } $.
The estimate \eqref{invbdH2s7}  follows.   
\end{proof}

Knowing a bound  for $\| u \|_{L^2_{\eps, \delta}}$, Proposition~\ref{estHs} implies by induction the following final result. 

\begin{prop}\label{prop73}
There are  constants $C$,  $\eps_0 > 0 $ and  $\delta_0 > 0$   
and  for  $s \ge 3$  there is a constant $C_s$
such that 
for $\eps \in ]0, \eps_0]$, $\delta \in [0, \delta_0]$,     
 $f \in H^{s+1}_{\eps, \delta} $, $g \in H^{s}_{\eps, \delta} $,
$\tilde U \in H^{s+1}_{\eps, \delta}$,
and $U \in H^s_{\eps, \delta}$ satisfying  
  \eqref{linpert}, \eqref{phasecond}, and \eqref{cond1first},
one has 
 \begin{equation}\label{invbdHs7}
\begin{aligned}
\big\|  U   \big\|_{H^s_{\eps, \delta} }&\le 
C \eps^{-1}\big( 
\big\| \tilde U \|_{H^{s+1}_{\eps, \delta} }
\big| F |_{{s_0+2}}
+
\big\| F \|_{{s+1}}\big)\\
&=
C \eps^{-1}\big( 
\big\| \tilde U \|_{H^{s+1}_{\eps, \delta} }
(\big\| f \|_{H^{s_0+3}_{\eps, \delta} }
+
\big\| g \|_{H^{s_0+2}_{\eps, \delta} })
+
(\big\| F \|_{H^{s+2}_{\eps, \delta} }
+\big\| g \|_{H^{s+1}_{\eps, \delta} }  \big)),
\\
\end{aligned}
\end{equation}
 
\end{prop}

\begin{rem}\label{whynm}
\textup{
The loss of derivative on $\tilde U$ comes from
the conservative form of the linearized equations,
through the microscopic energy estimates on the solution.
A similar loss in derivative may be seen
in the resolvent equation for
linear hyperbolic equations in conservative form,
$\lambda U+ (A(\tilde U)u)'=f$;
see \cite{TZ} for further discussion.
We could avoid this by writing the differentiated
equations in quasilinear form, but this would
prevent us from integrating back to carry out linearized
Chapman--Enskog estimates.
That is, the loss of derivatives is due to a subtle incompatibility
between the integrated form needed for linearized Chapman--Enskog estimates
and the nonconservative (quasilinear) form  needed
for optimal energy estimates with no loss of derivative.
}
\end{rem}
%

\section{Existence for the linearized problem}\label{existence}
To complete the proof of Proposition \ref{invprop}, it remains
to demonstrate existence for the linearized problem.
This can be carried out as in \cite{MZ2} by the vanishing viscosity
method, with viscosity coefficient $\eta>0$, obtaining
existence for each positive $\eta$ by standard boundary-value theory,
and noting that our previous A Priori bounds 
\eqref{invbdHs7} persist under regularization for
sufficiently small viscosity $\eta>0$, so that we can obtain
a weak solution in the limit by extracting a weakly convergent
subsequence.
We omit these details, referring the reader to
Section 8, \cite{MZ2}.
The asserted estimates then follow in the limit by continuity.

\appendix
\section{A Nash--Moser Theorem with losses}\label{NM}


For completeness, we give in this appendix the parameter-dependent
Nash--Moser theory developed in \cite{TZ}, specialized for clarity
to the present, Hilbert space, setting.
The main novelty of this treatment is to allow losses
of powers of the parameter $\eps\to 0$ in the linearized solution operator.
For a proof of this result, see \cite{TZ};
for a more general discussion of Nash--Moser iteration methods,
see \cite{H,AG,XSR}, and references therein.


 Consider two families of Banach spaces $\{E_s\}_{s \in \R},$ $\{F_s\}_{s \in \R},$ and a family of equations
 \begin{equation} \label{0}
  \Phi^\e(u^\e) = 0, \qquad u^\e \in E_s,
 \end{equation}
 indexed by $\e \in (0,1),$ where for all $\e,$ 
 \begin{equation} \label{0.1} \Phi^\e \in C^2(E_s,F_{s-m}), \qquad \mbox{ for all $s \leq \bar s,$}\end{equation}
 for some $ m \geq0$ and some $\bar s \in \R.$
 
  Let $| \cdot |_s$ denote the norm
in $E_s$ and $\| \cdot \|_s$ denote the norm in $F_s.$ 
The norms $|\cdot|_s$ and $\| \cdot\|_s$ may be $\e$-dependent
(as in our application here).
We assume that the embeddings
   \begin{equation} \label{sp0} E_{s'} \hookrightarrow E_{s}, \qquad F_{s'} \hookrightarrow F_{s}, \qquad s \leq s',\end{equation}
   hold, and have norms less than one:
   \begin{equation} \label{sp1} |\cdot|_s \leq |\cdot |_{s'}, \quad \| \cdot \|_s \leq \|\cdot \|_{s'}, \qquad s \leq s'.\end{equation}
  We assume the interpolation property\footnote{In \eqref{sp2} and below, $| u |_s \lesssim | v|_{s'}$ stands for $| u|_s \leq C |v|_{s'},$ for some $C > 0$ depending on $s$ and $s'$ but not on $\e,$ nor on $u$ and $v.$}:
   \begin{equation} \label{sp2}
    | \cdot|_{s + \s} \lesssim |\cdot|_{s}^{\frac{\s' - \s}{\s'}} |\cdot|_{s + \s'}^{\frac{\s}{\s'}}, \qquad 0 < \s < \s'.
   \end{equation}
 We assume in addition the existence of a family of regularizing operators
  $$S_\theta: \quad E_s \to E_s, \qquad \theta  > 0,$$
   such that for all $s \leq s,',$
 \begin{equation} \label{3}
   | S_\theta u - u |_{s} \lesssim \theta^{s - s'} | u |_{s'}.
 \end{equation}
\begin{equation} \label{4}
   | S_\theta u |_{s'} \lesssim \theta^{s' - s} | u |_{s}.
 \end{equation}

   \begin{ass} \label{ass1} For some $s_0 \in \R,$ some $\g_0 \geq 0,$ for all $s$ such that 
 $$s_0 + m \leq s + m \leq \bar s,$$
   for all $u,v,w \in E_{s+m},$
  \begin{equation} \label{1}
  \| \Phi^\e(u)\|_s \leq C_0 (1 + | u |_{s+m} + |u|_{s_0 + m}|u|_s); 
  \end{equation} 
  \begin{equation}
  \| (\Phi^\e)'(u) \cdot v\|_s  \leq  C_0 (|v|_{s +m} + |v|_{s_0 + m} |u|_{s+m}), \label{1'} 
  \end{equation}
  and
  \begin{equation} \begin{aligned}
  \label{1''} \| (\Phi^\e)''(u) \cdot (v, w)\|_s \leq C_0 & \big(|v|_{s_0 + m} |w|_{s + m} + |v|_{s + m} |w|_{s_0 + m} \\ & + |u|_{s + m}|v|_{s_0 + m} |w|_{s_0 + m}\big)
  \end{aligned}\end{equation}
    where $C_0 = C_0(\e, |u|_{s_0 + m})$ satisfies
 \begin{equation} \label{supc012}
  \sup_{\e} \sup_{|u|_{s_0 + m} \lesssim \e^{\gamma_0}} C_0  < + \infty.
  \end{equation} 
   \end{ass}

\begin{ass} \label{ass2} 
For some $\gamma \geq 0, r \geq 0, r' \geq 0,$ for all $s$ such that
 \begin{equation} \label{ass-s-0}
  s_0 + m + \max(r,r') \leq s + \max(r,r') \leq \bar s,
  \end{equation}
   for all $u \in E_{s + r}$ such that
\begin{equation} \label{cond1}
 | u |_{s_0 + m} \lesssim \e^{\gamma},
 \end{equation}
 the map $(\Phi^\e)'(u): E_{s+m} \to F_s$ has a right inverse $\Psi^\e(u):$
  $$ (\Phi^\e)'(u) \Psi^\e(u) = {\rm Id}: \quad F_s \to F_s,$$
   satisfying, for all $\phi \in F_{s+r'},$
 \begin{equation} \label{2}
  | \Psi^\e(u) \phi|_s \leq \e^{-1} C( \|\phi\|_{s_0 + m + r'} | u|_{s+r} + \|\phi\|_{s + r'}),
 \end{equation}
 where $C$ is a non-decreasing function of its arguments $s$ and $|u|_{s_0 + m + r}.$ 
\end{ass}

\begin{ass} \label{ass-wkb} There holds the bound 
\begin{equation} \label{wkb3}
 \|\Phi^\e(0)\|_{s} \lesssim \e^{k}, \end{equation}
 for some $k$ and $s$ satisfying
 \begin{equation} \label{ass-wkb-cond-k}
  \max(2, 1+ \g_0, 1 + \g) < k,
  \end{equation}
  \begin{equation} \label{ass-bar-s}
  C(k) \leq \bar s - s_0 - m,
  \end{equation}
 where $C(k)$ is a certain positive function (see \cite{TZ})
and $ s \in [s_0 + m, \bar s - C(k)].$
  \end{ass}

  \begin{theo}[Existence] \label{th1} Under Assumptions {\rm \ref{ass1},} {\rm \ref{ass2}} and {\rm \ref{ass-wkb}}, for $\e$ small enough, there exists a real sequence $\theta_j^\e,$ satisfying $\theta_j^\e \to +\infty$ as $j \to +\infty$ and $\e$ is held fixed, such that the sequence
 $$ u^\e_0 := 0, \qquad u^\e_{j+1} := u_j^\e + S_{\theta_j^\e} v_j^\e, \qquad v_j^\e := - \Psi^\e(u_j^\e) \Phi^\e(u_j^\e),$$
  is well defined and converges, as $j \to \infty$ and $\e$ is held fixed, to a solution $u^\e$ of \eqref{0} 
in $s +m$ norm, which satisfies the bound
 \begin{equation} \label{est-th}
 | u^\e |_{s} \lesssim \e^{k -1}. 
 \end{equation}
\end{theo}

 \begin{theo}[Uniqueness] \label{th2}
Under Assumptions {\rm \ref{ass1},} {\rm \ref{ass2}} and {\rm \ref{ass-wkb}}, 
for $\e$ small enough, 
if $(\Phi^\e)'$ is invertible, i.e., $\Psi^\e$ is also a left inverse, then
the solution described in Thm \ref{th1} is unique in a ball of radius 
$o(\eps^{\max(1,\gamma_0,\gamma)})$ in 
$s_0 +2m + r'$ norm.
More generally, if $\hat u^\e$ is a second solution within this ball, then
$(\hat u^\e- u^\e)$ is approximately tangent to $\kernel (\Phi^\e)'(u^\e)$, in the
sense that its distance in $s_0$ norm from $\kernel (\Phi^\e)'(u^\e)$ is
$o(|\hat u^\e-u^\e|_{s_0 })$.
In particular, if $\kernel (\Phi^\e)'(u^\e)$ is finite-dimensional,
then  $u$ is the unique solution in the ball satisfying the additional
``phase condition''
\be\label{phasecond}
\Pi_{\kernel (\Phi^\e)'(u^\e)} (\hat u^\e-u^\e)=0,
\ee
where 
$\Pi_{\kernel (\Phi^\e)'(u^\e)}$ is any uniformly bounded projection 
onto $\kernel (\Phi^\e)'(u^\e)$ 
(in a Hilbert space, any orthogonal projection onto $\kernel (\Phi^\e)'(u^\e)$).
\end{theo}


\end{document}